\documentclass[12pt, a4paper]{amsart}
\usepackage[utf8]{inputenc}
\usepackage[T1]{fontenc}
\usepackage{amsmath, amssymb, amsthm}
\usepackage{geometry}
\usepackage{hyperref}
\newtheorem{theorem}{Theorem}[section]
\newtheorem{lemma}[theorem]{Lemma}
\newtheorem{proposition}[theorem]{Proposition}
\newtheorem{corollary}[theorem]{Corollary}
\newtheorem{problem}[theorem]{Problem}
\theoremstyle{definition}
\newtheorem{definition}[theorem]{Definition}
\newtheorem{remark}[theorem]{Remark}

\theoremstyle{plain}
\newtheorem*{thmA}{Theorem A}
\newtheorem*{thmB}{Theorem B}
\newtheorem*{thmC}{Theorem C}

\newcommand{\Bl}{\mathcal{B}(D,\ell_2)}
\newcommand{\Blz}{\mathcal{B}_{0,c}(D,\ell_2)}
\newcommand{\MS}{\mathcal{M}_{S_1}(2,1)}
\newcommand{\Minf}{\mathcal{M}_{\infty}(2,1)}
\newcommand{\RR}{\mathcal{R}(2,1)}
\newcommand{\CC}{\mathcal{C}(2,1)}
\newcommand{\Mone}{\mathcal{M}_1}
\newcommand{\N}{\mathbb{N}}
\newcommand{\T}{\mathbb{T}}
\newcommand{\Z}{\mathbb{Z}}
\newcommand{\e}{\mathrm{e}}
\newcommand{\lsim}{\lesssim}
\newcommand{\esim}{\eqsim}

\title{Structural Properties of the Köthe Dual of the Matricial Bloch Space}
\author{Liviu-Gabriel Marcoci} 

\address{Department of Mathematics and Computer Science, Technical University of Civil Engineering Bucharest, RO-020396 Bucharest, ROMANIA}
\email{liviu.marcoci@utcb.ro}
\date{\today}

\begin{document}

\begin{abstract}
We study the K\"othe dual $\mathcal{B}(D,\ell_2)^K$ of the matricial
Bloch space. A 2015 conjecture [Publ. Math. Debrecen \textbf{87}
(2015), 351--370] proposed that this space coincides with the dyadic
mixed-norm space determined by the operator norms of the diagonals. We
disprove the conjecture by revealing a structural obstruction:
membership in $\mathcal{B}(D,\ell_2)^K$ is sensitive to the placement
of the entries within the diagonals and cannot be detected from
diagonal data alone; in particular the trace-norm variant fails as
well. However, testing against Toeplitz matrices exactly recovers the
trace-norm variant, a matricial analogue of the Anderson--Shields
theorem, proved via analytic majorants. Finally, we establish
two-sided estimates: row-wise and column-wise $\ell(2,1)$ conditions
are sufficient, while the trace-norm dyadic condition is necessary;
the latter inclusion is strict.
\end{abstract}
\maketitle



\section{Introduction and main results}\label{sec:intro}

A guiding principle of matriceal harmonic analysis, developed
systematically in the monograph of Persson and Popa \cite{PP}, is that
an upper triangular infinite matrix $A=(a_{j,m})_{j,m\ge1}$ may be
regarded as an analogue of an analytic function on the unit disc $D$,
the diagonals $A_k$ ($k\ge0$) playing the role of the Taylor
coefficients. Many classical function spaces admit matricial
analogues, and a recurring theme is to decide which properties of a
matrix are visible from the sequence of its diagonal norms alone.

The scalar model for the problem studied here is the theorem of
Anderson and Shields \cite{AS} on the Bloch space
\[
\mathcal{B}=\{f \text{ analytic on } D:
\|f\|_{\mathcal{B}}:=\sup_{|z|<1}(1-|z|^2)|f'(z)|+|f(0)|<\infty\}.
\]
Regarding analytic functions as sequences of Taylor coefficients, the
K\"othe dual of a sequence space $X$ is
$X^K=(X,\ell_1)$, the space of coefficient multipliers from $X$ into
$\ell_1$. Anderson and Shields proved that
\begin{equation}\label{eq:AS}
\mathcal{B}^K=(\mathcal{B},\ell_1)=\ell(2,1),
\end{equation}
where $\ell(2,1)$ is the dyadic mixed-norm space recalled in
Section~\ref{sec:prelim}; see \cite[Section~3]{AS}, and
\cite{BST,JVA,Pav} for the surrounding scalar theory.

For the infinite matrices $A=(a_{j,m})$ and $B=(b_{j,m})$, the Schur product is defined as
\[
A*B=(a_{j,m}b_{j,m})_{j,m}.
\]
The role of $\ell_1$ is played by the
space $\Mone$ of matrices with absolutely summable entries, and the
K\"othe dual of a space $X$ of upper triangular matrices is
\begin{equation}\label{eq:kothe}
X^K=\Bigl\{A \text{ upper triangular}:\
\sum_{j,m}|a_{j,m}b_{j,m}|<\infty\ \text{ for every }B\in X\Bigr\}.
\end{equation}
 The matricial Bloch space $\Bl$
(Definition~\ref{def:bloch}) was introduced by Popa \cite{Po} and is
studied systematically in \cite{PP}. In \cite{MMPP}, Marcoci, Marcoci,
Persson and Popa characterized the Schur multipliers from the Schatten
classes $S_p$, $1<p<\infty$, into $\Mone$ that is, the K\"othe
duals of the Schatten classes \cite[Theorem~2.3]{MMPP}. They also proved
that the K\"othe duals of the space $\mathcal{I}(\ell_2)$ (studied in
\cite{Po,PP}) and of the little space $\Blz$ recalled in
Section~\ref{sec:prelim} coincide with the largest solid subspaces of
$\Bl$ and of $\mathcal{I}(\ell_2)$, respectively
\cite[Theorem~4.4]{MMPP}. For $\Bl$ itself, the following conjecture,
influenced by the Anderson--Shields identity \eqref{eq:AS}, is
formulated at the end of \cite{MMPP}:
\begin{equation}\label{eq:conj}
\Bl^K\ \overset{?}{=}\ \Minf ,
\end{equation}
where $\Minf$ is the dyadic mixed-norm space determined by the
operator norms of the diagonals (Definition~\ref{def:mixed}; the space
is denoted $\mathcal{M}(2,1)$ in \cite{MMPP}). Since $\Mone$ is the
trace-class analogue of $\ell_1$, a second natural candidate arises by
using instead the trace norms of the diagonals: writing $\Phi(A_k)$
for a norm of the $k$-th diagonal, the candidate spaces are
\[
\Bigl\{A:\ \sum_{n\ge0}\Bigl(\sum_{k\in
I_n}\Phi(A_k)^2\Bigr)^{1/2}<\infty\Bigr\},
\]
with $\Phi$ the operator norm (giving $\Minf$) or the trace norm
(giving $\MS$).

The first purpose of this paper is to show that the conjecture
\eqref{eq:conj} fails, that its trace-norm variant fails as well, and
that the failure is structural: \emph{membership in $\Bl^K$ cannot be
detected from the diagonals at all}. Throughout,
$(I_n)_{n\ge0}$ denotes the dyadic decomposition of $\{0,1,2,\dots\}$
fixed in Section~\ref{sec:prelim}, and a functional $\Phi$ on diagonal
matrices is called \emph{permutation-invariant} if it depends only on
the multiset of the moduli of the entries; every Schatten norm and
every symmetric norm is of this type.

\begin{thmA}[placement dichotomy; Theorem~\ref{thm:dichotomy} and
Corollary~\ref{cor:impossible}]
There exist upper triangular matrices $A^{(1)},A^{(2)}\in\MS$ whose
diagonals coincide up to a permutation of their entries, such that
\[
A^{(2)}\in\Bl^K\qquad\text{and}\qquad A^{(1)}\notin\Bl^K .
\]
Consequently, for every permutation-invariant functional $\Phi$ on
diagonals and every set $S$ of sequences,
\[
\bigl\{A:\ (\Phi(A_k))_{k\ge0}\in S\bigr\}\ \neq\ \Bl^K ;
\]
in particular $\Minf\not\subset\Bl^K$ and $\MS\not\subset\Bl^K$.
\end{thmA}

In particular the conjecture \eqref{eq:conj} is disproved;
Proposition~\ref{prop:warmup} records an even simpler counterexample to
\eqref{eq:conj} alone, valid under either of the two dyadic conventions
in use (see Section~\ref{sec:prelim}). The matrices $A^{(1)},A^{(2)}$
of Theorem~A are completely explicit: both have the
entry $1/k$ on the $k$-th diagonal for every $k\ge1$, placed at
position $(k,2k)$ in the first case and at position $(1,1+k)$ in the
second. The failure persists for the little matricial Bloch space
(Remark~\ref{rem:littlebloch}).

The second purpose is to prove positive results bounding $\Bl^K$
from both sides. Let $\RR$ and $\CC$ denote the spaces of upper
triangular matrices whose rows, respectively columns, have summable
$\ell(2,1)$ norms (Definition~\ref{def:rowcol}).

\begin{thmB}[two-sided estimate; Theorem~\ref{thm:sufficient},
Theorem~\ref{thm:necessary} and Corollary~\ref{cor:sandwich}]
\[
\RR+\CC\ \subset\ \Bl^K\ \subset\ \MS ,
\]
with quantitative bounds, and the second inclusion is strict.
Moreover, neither the row condition nor the column condition is
necessary: $\Bl^K\setminus\RR$ and $\Bl^K\setminus\CC$ are both
nonempty.
\end{thmB}

The proof of the right-hand inclusion yields more: Toeplitz tests
already characterize $\MS$, which is the exact matricial analogue of the
Anderson--Shields identity \eqref{eq:AS}. Here $T(f)$ denotes the
upper triangular Toeplitz matrix whose diagonals are the Taylor
coefficients of $f$, and
\[
\sigma(A):=\sup\Bigl\{\sum_{k\ge0}\|A_k\|_{S_1}\,|\widehat f(k)|:\
f\in\mathcal{B},\ \|f\|_{\mathcal{B}}\le1\Bigr\}.
\]

\begin{thmC}[Toeplitz tests; Theorem~\ref{thm:toeplitzdual}]
For an upper triangular matrix $A$ the following are equivalent:
\begin{itemize}
\item[\rm(a)] $\sum_{j,m}|a_{j,m}b_{j,m}|<\infty$ for every Toeplitz
matrix $B\in\Bl$;
\item[\rm(b)] $A\in\MS$.
\end{itemize}
Moreover $\sigma(A)\esim\|A\|_{\MS}$, with absolute implied constants.
\end{thmC}

Theorem~A shows that passing from Toeplitz tests to all of $\Bl$
destroys the description in Theorem~C, and that no repair within the
class of diagonal descriptions is possible. A further consequence of
Theorem~C worth recording is that $\Bl^K$ contains no nonzero Toeplitz
matrix (Remark~\ref{rem:notoeplitz}), in sharp contrast with $\Bl$
itself, which contains an isometric Toeplitz copy of the scalar Bloch
space (Lemma~\ref{lem:toeplitz}).

The necessity parts of Theorems~B and~C are proved by testing against
Toeplitz matrices whose symbols are lacunary sums of \emph{
analytic majorants}: polynomials with spectrum in a fixed dyadic band
whose coefficients dominate prescribed moduli and whose supremum norm
is controlled by their $\ell_2$ norm. These are obtained
from the theorem of de~Leeuw, Kahane and Katznelson \cite{dLKK} by a
band-limiting argument (Lemma~\ref{lem:majorant}); analytic majorants
of this type go back to Kislyakov \cite{Kis}. We note that the
original proof of \eqref{eq:AS} in \cite{AS} proceeds instead through
Khintchine's inequality and Abel duality; we have preferred a direct,
self-contained argument adapted to the matricial pairing.

Related Schur-multiplier problems for matrices with operator entries
are studied in \cite{BG}; the phenomena there are of a different
nature, but they confirm that matricial multiplier problems are
sensitive to more than the diagonal norms.

The paper is organized as follows. Section~\ref{sec:prelim} fixes
notation and contains three lemmas used throughout.
Section~\ref{sec:failure} proves Theorem~A,
Section~\ref{sec:sufficient} the left-hand half of Theorem~B, and
Section~\ref{sec:necessary} Theorem~C and the right-hand half of
Theorem~B. Section~\ref{sec:problems} collects open problems.

\section{Preliminaries}\label{sec:prelim}

\subsection*{Notation}

For nonnegative quantities $u,v$ we write $u\lsim v$ if $u\le Cv$ for
an absolute constant $C\in(0,\infty)$, independent of all parameters
involved (in particular of the matrices, indices and radii); $u\esim v$
means $u\lsim v$ and $v\lsim u$.

All matrices are infinite matrices $A=(a_{j,m})_{j,m\ge1}$ with complex
entries and are assumed \emph{upper triangular}: $a_{j,m}=0$ for
$m<j$. For $k\ge0$ the $k$-th diagonal $A_k$ is the matrix agreeing
with $A$ on the entries $(j,j+k)$, $j\ge1$, and vanishing elsewhere.
We write
\[
a_k^{\,l}:=a_{l,\,l+k}\qquad(l\ge1,\ k\ge0)
\]
for the $l$-th entry of the $k$-th diagonal, so that $l$ labels the row
and $m=l+k$ the column of the entry. The row and column sequences of
$A$ are
\[
R_j(A):=(a_k^{\,j})_{k\ge0}\quad(j\ge1),\qquad
C_m(A):=(a_k^{\,m-k})_{0\le k\le m-1}\quad(m\ge1);
\]
each column of an upper triangular matrix is finite.

A diagonal $A_k$ is identified with the diagonal-type operator it
induces on $\ell_2$; thus $\|A_k\|_{B(\ell_2)}=\sup_l|a_k^{\,l}|$ and
$\|A_k\|_{S_1}=\sum_l|a_k^{\,l}|$, where $S_1$ is the trace class.
The Schur product is $A*B=(a_{j,m}b_{j,m})_{j,m}$ and
\[
\Mone:=\Bigl\{A:\ \|A\|_{\Mone}:=\sum_{j,m}|a_{j,m}|
=\sum_{k\ge0}\|A_k\|_{S_1}<\infty\Bigr\}.
\]
For a space $X$ of upper triangular matrices the K\"othe dual $X^K$ is
defined by \eqref{eq:kothe}; equivalently, $A\in X^K$ if and only if
$A*B\in\Mone$ for every $B\in X$. The space $X^K$ is always
\emph{solid}: if $A\in X^K$ and $|c_{j,m}|\le|a_{j,m}|$ for all $j,m$,
then $C\in X^K$.

\subsection*{Dyadic blocks and mixed norms}
We fix the dyadic decomposition
\[
I_0:=\{0,1\},\qquad I_n:=\{k\in\N:\ 2^n\le k<2^{n+1}\}\quad(n\ge1),
\]
so that every integer $k\ge0$ lies in exactly one block. In
\cite{AS,MMPP} the blocks are given by the second formula for all
$n\ge0$, the sequences in \cite{AS} being indexed from $k=1$; under
that convention the index $k=0$, the main diagonal, belongs to no
block. The convention above covers every $k\ge0$ and is the one used
throughout; the counterexamples of Section~\ref{sec:failure} are valid
under either convention (see Proposition~\ref{prop:warmup}). For a
scalar sequence $x=(x_k)_{k\ge0}$ we write
\[
\|x\|_{\ell(2,1)}:=\sum_{n\ge0}\Bigl(\sum_{k\in
I_n}|x_k|^2\Bigr)^{1/2},
\]
and note that
\begin{equation}\label{eq:harmonic}
\bigl\|(1/k)_{k\ge1}\bigr\|_{\ell(2,1)}
\le1+\sum_{n\ge1}\Bigl(2^n\cdot2^{-2n}\Bigr)^{1/2}
=1+\sum_{n\ge1}2^{-n/2}<\infty ,
\end{equation}
a computation used  below.

\begin{definition}\label{def:mixed}
For an upper triangular matrix $A$ set
\[
\|A\|_{\MS}:=\bigl\|(\|A_k\|_{S_1})_{k\ge0}\bigr\|_{\ell(2,1)},\qquad
\|A\|_{\Minf}:=\bigl\|(\|A_k\|_{B(\ell_2)})_{k\ge0}\bigr\|_{\ell(2,1)},
\]
and let $\MS$ and $\Minf$ be the corresponding spaces of matrices with
finite norm. It it easy to see that $\MS\subset \Minf$.
\end{definition}

\begin{definition}\label{def:rowcol}
For an upper triangular matrix $A$ set
\[
\|A\|_{\RR}:=\sum_{j\ge1}\|R_j(A)\|_{\ell(2,1)},\qquad
\|A\|_{\CC}:=\sum_{m\ge1}\|C_m(A)\|_{\ell(2,1)},
\]
the finite sequences $C_m(A)$ being extended by zero, and let $\RR$,
$\CC$ be the corresponding spaces and
$\RR+\CC=\{A'+A'':A'\in\RR,\ A''\in\CC\}$.
\end{definition}

All four quantities are norms and the corresponding spaces are Banach
spaces; the routine verifications are omitted. Since
$\|A_k\|_{S_1}=\sum_j|a_k^{\,j}|=\sum_m|a_k^{\,m-k}|$, Minkowski's
inequality in $\ell_2(I_n)$ gives
\begin{equation}\label{eq:minkowski}
\|A\|_{\MS}\le\min\bigl(\|A\|_{\RR},\,\|A\|_{\CC}\bigr).
\end{equation}

\subsection*{The matricial Bloch space}
For $0\le r<1$ let
\[
A'(r):=\sum_{k\ge1}kA_kr^{k-1},
\]
understood entrywise and required to define a bounded operator on
$\ell_2$ for each $r$.

\begin{definition}\label{def:bloch}
The matricial Bloch space $\Bl$ consists of all upper triangular
matrices $A$ with
\[
\|A\|_{\Bl}:=\sup_{0\le r<1}(1-r^2)\,\|A'(r)\|_{B(\ell_2)}
+\|A_0\|_{B(\ell_2)}<\infty .
\]
\end{definition}

$(\Bl,\|\cdot\|_{\Bl})$ is a Banach space; see \cite{PP,MMPP}. We also
consider the little space
\[
\Blz:=\Bigl\{A\in\Bl:\ \text{each }A_k\text{ is compact and }
\lim_{r\to1^-}(1-r^2)\|A'(r)\|_{B(\ell_2)}=0\Bigr\},
\]
introduced in \cite{MPPP}; see also \cite{MMPP,PP}.
For $f(z)=\sum_{k\ge0}c_kz^k$ analytic on $D$, $T(f)$ denotes the upper
triangular Toeplitz matrix with $t_k^{\,l}=c_k$ for all $l$.

\subsection*{Three lemmas}

The material of this subsection is partly known, and we indicate the
sources. Lemma~\ref{lem:diagproj} and Lemma~\ref{lem:toeplitz} (i) are folklore in the Schur
multiplier literature (see, e.g., \cite{Ben,PP}).
Lemma~\ref{lem:coeff}(i) appears, with an unspecified constant, in the
proof of \cite[Theorem~6]{Po}, and its analogue for the little space is
used in the proof of \cite[Theorem~3.4]{MMPP}; the identification of
Lemma~\ref{lem:toeplitz} goes back to \cite{Po} (see also \cite{PP}).
Short proofs are nevertheless included, both to keep the paper
self-contained and because the explicit constants enter the arguments
of Section~\ref{sec:necessary}. By contrast, the row and column dyadic
estimates of Lemma~\ref{lem:coeff}(ii), which are the technical core of
our positive results, extend to Bloch \emph{matrices} the classical
scalar estimate of Anderson and Shields (see the proof of
\cite[Lemma~8]{AS}) and we have not found them in the literature in this form.

\begin{lemma}[diagonal projections]\label{lem:diagproj}
It yields $\|M_k\|_{B(\ell_2)}\le\|M\|_{B(\ell_2)}$ for every $M\in B(\ell_2)$
and every $k\in\Z$.
\end{lemma}

\begin{proof}
Let $D_t=\operatorname{diag}(\e^{2\pi ilt})_{l\ge1}$, so that the
$(j,m)$ entry of $D_tMD_t^*$ is $m_{j,m}\e^{-2\pi i(m-j)t}$ and hence
\[
M_kx=\int_0^1\e^{2\pi ikt}\,D_tMD_t^*x\,dt\qquad(x\in\ell_2),
\]
the integral converging in $\ell_2$ because $t\mapsto D_tx$ is
continuous for each fixed $x$, by dominated convergence. Consequently
$\|M_kx\|\le\|M\|\,\|x\|$ for every $x$, that is, $\|M_k\|\le\|M\|$.
\end{proof}

\begin{lemma}[coefficient estimates]\label{lem:coeff}
Let $B\in\Bl$.
\begin{itemize}
\item[\rm(i)] $\|B_k\|_{B(\ell_2)}\le\e\,\|B\|_{\Bl}$ for every
$k\ge0$; in particular $|b_k^{\,l}|\le\e\,\|B\|_{\Bl}$ for all $k,l$.
\item[\rm(ii)] The following estimates hold uniformly over the block index $n\ge0$, the row index $j\ge1$, and the column index $m\ge1$:
\[
\Bigl(\sum_{k\in I_n}|b_k^{\,j}|^2\Bigr)^{1/2}\lsim\|B\|_{\Bl}
\qquad\text{and}\qquad
\Bigl(\sum_{k\in I_n,\ k<m}|b_k^{\,m-k}|^2\Bigr)^{1/2}
\lsim\|B\|_{\Bl}.
\]
\end{itemize}
\end{lemma}

\begin{proof}
(i) The case $k=0$ is part of the definition. For $k\ge1$,
Lemma~\ref{lem:diagproj} applied to $B'(r)$ gives
\[
k\,\|B_k\|\,r^{k-1}\le\|B'(r)\|\le\|B\|_{\Bl}/(1-r).
\]
The case $k=1$ follows by taking $r=0$; for $k\ge2$ take $r=1-1/k$ and use
$(1-1/k)^{k-1}\ge\e^{-1}$.

(ii) The function $h(x)=\ln(1-x)+(2\ln2)x$ is concave with
$h(0)=h(1/2)=0$, so
\begin{equation}\label{eq:logineq}
\ln(1-x)\ge-(2\ln2)\,x\qquad(0<x\le\tfrac12).
\end{equation}
Fix $j\ge1$ and $n\ge1$, and denote $r_n=1-2^{-n}$. 

The $\ell_2$ norm of
the $j$-th row of a bounded operator $M$ equals $\|M^*e_j\|_2\le\|M\|$;
the $j$-th row of $B'(r)$ has the entry $k\,b_k^{\,j}r^{k-1}$ in column
$j+k$, thus
\begin{equation}\label{eq:rowtest}
\Bigl(\sum_{k\ge1}k^2|b_k^{\,j}|^2r^{2(k-1)}\Bigr)^{1/2}
\le\|B'(r)\|\le\frac{1}{1-r}\|B\|_{\Bl}.
\end{equation}
For $k\in I_n$ we have $k\ge2^n$ and, by \eqref{eq:logineq} with
$x=2^{-n}$,
\[
r_n^{\,k-1}\ge r_n^{\,2^{n+1}}
=\exp\bigl(2^{n+1}\ln(1-2^{-n})\bigr)\ge2^{-4},
\]
so that $k\,r_n^{\,k-1}\gtrsim2^n$ on $I_n$.

Since $1-r_n=2^{-n}$,
evaluating \eqref{eq:rowtest} at $r=r_n$ yields
\[
2^n\bigl(\sum_{k\in I_n}|b_k^{\,j}|^2\bigr)^{1/2}
\lsim2^n\|B\|_{\Bl},
\]
which is the row estimate for $n\ge1$; the block
$I_0$ is covered by (i). 

The column estimate is identical: the
$\ell_2$ norm of the $m$-th column of $M$ is 
\[
\|Me_m\|_2\le\|M\|,
\]
and
the $m$-th column of $B'(r)$ has the entry $k\,b_k^{\,m-k}r^{k-1}$ in
row $m-k$, $1\le k\le m-1$.
\end{proof}

\begin{lemma}[Toeplitz matrices; cf.~{\cite{Po,PP}}]\label{lem:toeplitz}
Let $f(z)=\sum_{k\ge0}c_kz^k$ be analytic on $D$.
\begin{itemize}
\item[\rm(i)] If $f\in H^\infty$ then $T(f)\in B(\ell_2)$ and
$\|T(f)\|_{B(\ell_2)}=\|f\|_{H^\infty}$.
\item[\rm(ii)] $T(f)\in\Bl$ if and only if $f\in\mathcal{B}$, and
$\|T(f)\|_{\Bl}=\|f\|_{\mathcal{B}}$.
\end{itemize}
\end{lemma}

\begin{proof}
(i) In the monomial basis of $H^2\cong\ell_2$ the multiplication
operator $M_f$ has the lower triangular matrix
$(\widehat f(j-m))_{j,m}$ and $\|M_f\|=\|f\|_{H^\infty}$. The matrix
$T(f)$ is its transpose, and transposition preserves the operator
norm, since $M^{\mathsf T}=JM^*J$ for the coordinatewise conjugation
$J$, an isometric conjugate-linear involution.

(ii) The diagonals of $T(f)$ are $c_kI$; since
$zf'(rz)=\sum_{k\ge1}kc_kr^{k-1}z^k$, it follows that
$T(f)'(r)=T(g_r)$ with $g_r(z):=zf'(rz)$.
By (i),
\[
\|T(f)'(r)\|_{B(\ell_2)}=\sup_{|z|\le1}|zf'(rz)|
=\max_{|w|=r}|f'(w)|,
\]
and $\|T(f)_0\|=|f(0)|$; taking the supremum gives the claim.
\end{proof}

\section{Failure of diagonal descriptions}\label{sec:failure}

We begin with an elementary observation that already disproves the
conjecture \eqref{eq:conj}; it also motivates the trace-norm candidate,
which it does not exclude.

\begin{proposition}\label{prop:warmup}
We have that $\Minf\not\subset\Bl^K$; in particular the conjecture \eqref{eq:conj}
of \cite{MMPP} is false.
\end{proposition}

\begin{proof}
Let $A$ having the main diagonal 
\[ 
a_0^{\,l}=1/l, \quad l\ge1
\]and let $B$
be the identity. Then 
\[
\|A\|_{\Minf}=\|A_0\|_{B(\ell_2)}=1
\] 
and
$\|B\|_{\Bl}=1$ (as $B'(r)=0$), while
\[
\sum_{j,m}|a_{j,m}b_{j,m}|=\sum_l1/l=\infty.
\]
The same matrices work
under the dyadic convention of \cite{MMPP}: there the index $k=0$
belongs to no block, so all block sums of $A$ vanish and $A$ lies in
the space $\mathcal{M}(2,1)$ of \cite{MMPP} trivially.
\end{proof}

Note that $\|A_0\|_{S_1}=\infty$ for the matrix above, so
Proposition~\ref{prop:warmup} does not exclude the
trace-norm candidate $\MS$. The main result of this section shows that
the trace-norm candidate fails as well, and that the failure is
structural.

\begin{theorem}[placement dichotomy]\label{thm:dichotomy}
Define upper triangular matrices $A^{(1)}$ and $A^{(2)}$ by
\[
(A^{(1)})_k^{\,l}=\frac{\delta_{l,k}}{k},\qquad
(A^{(2)})_k^{\,l}=\frac{\delta_{l,1}}{k}\qquad k\ge1,\ l\ge1,
\]
so that $A^{(1)}$ carries the entry $1/k$ at position $(k,2k)$ and
$A^{(2)}$ carries it at position $(1,1+k)$. Then:
\begin{itemize}
\item[\rm(i)] for every $k\ge0$ the diagonals $A^{(1)}_k$ and
$A^{(2)}_k$ coincide up to a permutation of their entries, and
$A^{(1)},A^{(2)}\in\MS$;
\item[\rm(ii)] $A^{(1)}\notin\Bl^K$: the matrix
$B:=\sum_{k\ge1}E_{k,2k}$ satisfies $\|B\|_{\Bl}\le2$, while the
pairing of $A^{(1)}$ with $B$ is the harmonic series;
\item[\rm(iii)] $A^{(2)}\in\Bl^K$, with
$\sum_{j,m}|(A^{(2)})_{j,m}\,b_{j,m}|\lsim\|B\|_{\Bl}$ for every
$B\in\Bl$.
\end{itemize}
\end{theorem}

\begin{proof}
(i) For each $k\ge1$ both diagonals consist of the single entry $1/k$,
placed at different positions; both $0$-th diagonals vanish. Since all
diagonals are rank one, $\|A^{(i)}_k\|_{S_1}=1/k$, and
$A^{(i)}\in\MS$ by \eqref{eq:harmonic}.

(ii) Here $E_{j,m}$ denotes the matrix unit. The matrix
\[
B'(r)=\sum_{k\ge1}kr^{k-1}E_{k,2k}
\]
has at most one nonzero entry in
every row and in every column, so
\[
\|B'(r)\|_{B(\ell_2)}=\sup_{k\ge1}kr^{k-1}.
\]
We claim that
\begin{equation}\label{eq:supbound}
(1-r^2)\,\sup_{k\ge1}kr^{k-1}\le2\qquad(0\le r<1).
\end{equation}
If $r\le\frac12$ then $kr^{k-1}\le k2^{1-k}\le1$ for all $k\ge1$. If
$r>\frac12$, then $\max_{x>0}xr^x=1/(\e\ln(1/r))$ gives
\[
kr^{k-1}=\frac1r\,kr^k\le\frac{1}{\e\,r\ln(1/r)}
\le\frac{2}{\e(1-r)},
\]
using $\ln(1/r)\ge1-r$, and consequently
$(1-r^2)kr^{k-1}\le2(1+r)/\e<2$. Since $B_0=0$, \eqref{eq:supbound}
gives $\|B\|_{\Bl}\le2$; and the pairing equals
$\sum_{k\ge1}1/k=\infty$.

(iii) Only the first row of $B$ enters the pairing. By the
Cauchy--Schwarz inequality on each dyadic block, by
Lemma~\ref{lem:coeff}(ii) applied to the row $j=1$, and by
\eqref{eq:harmonic},
\begin{align*}
\sum_{j,m}|(A^{(2)})_{j,m}\,b_{j,m}|
&=\sum_{k\ge1}\frac{|b_k^{\,1}|}{k}\\
&\le\sum_{n\ge0}\Bigl(\sum_{k\in I_n,\,k\ge1}k^{-2}\Bigr)^{1/2}
\Bigl(\sum_{k\in I_n}|b_k^{\,1}|^2\Bigr)^{1/2}\\
&\lsim\|B\|_{\Bl}. 
\end{align*}
The proof is complete. \qedhere
\end{proof}

\begin{corollary}[no diagonal description]\label{cor:impossible}
Let $\Phi$ be any permutation-invariant functional on diagonal
matrices, with values in $[0,\infty]$, and let $S$ be any set of
sequences. Then
\[
\bigl\{A\ \text{upper triangular}:\ (\Phi(A_k))_{k\ge0}\in S\bigr\}
\ \neq\ \Bl^K .
\]
In particular $\Bl^K$ is not a mixed-norm space of diagonal norms for
any norm on the diagonals and any sequence-space condition, and
$\MS\not\subset\Bl^K$, $\Minf\not\subset\Bl^K$.
\end{corollary}

\begin{proof}
By Theorem~\ref{thm:dichotomy}(i), $\Phi(A^{(1)}_k)=\Phi(A^{(2)}_k)$
for every $k$, so the displayed set contains both of
$A^{(1)},A^{(2)}$ or neither, while $\Bl^K$ contains exactly one of
them. The last assertions follow from
$A^{(1)}\in\MS\subset\Minf$ and Theorem~\ref{thm:dichotomy}(ii).
\end{proof}

\begin{remark}\label{rem:schatten}
Since the diagonals of $A^{(1)}$ are rank one, all Schatten norms
coincide on them; thus no mixed-norm space built from \emph{any}
Schatten norm $S_p$, $0<p\le\infty$, of the diagonals is contained in
$\Bl^K$.
\end{remark}

\begin{remark}[the little space]\label{rem:littlebloch}
The failure is not caused by the ``large'' matrices in $\Bl$. Let
$\widetilde B$ having the entry $1/\ln(k+2)$ and $\widetilde A$ the
entry $\ln(k+2)/k$ at position $(k,2k)$, $k\ge1$. Each diagonal of
$\widetilde B$ has finite rank, and by \eqref{eq:supbound}, for every
$K\ge1$,
\[
\limsup_{r\to1^-}\,(1-r^2)\,\|\widetilde B'(r)\|
\le\limsup_{r\to1^-}\Bigl[(1-r^2)\max_{k\le K}\frac{k}{\ln(k+2)}
+\frac{2}{\ln(K+2)}\Bigr]=\frac{2}{\ln(K+2)},
\]
so $\widetilde B\in\Blz$. On the other hand
$\|\widetilde A_k\|_{S_1}=\ln(k+2)/k$ gives dyadic block sums
$O\bigl((n+1)2^{-n/2}\bigr)$, so $\widetilde A\in\MS$, while the
pairing is $\sum_k1/k=\infty$. Hence $\MS\not\subset\Blz^K$, and the
placement dichotomy transfers verbatim to the little space. Since
$\Blz^K$ coincides with the largest solid subspace
$s(\mathcal{I}(\ell_2))$ by \cite[Theorem~4.4]{MMPP}, the example also
shows that $\MS\not\subset s(\mathcal{I}(\ell_2))$.
\end{remark}

\section{Sufficient conditions}\label{sec:sufficient}

The proof of Theorem~\ref{thm:dichotomy}(iii) suggests that row-wise
 and, by symmetry, column-wise conditions are the correct
source of membership in $\Bl^K$.

\begin{theorem}\label{thm:sufficient}
For every upper triangular matrix $A$ and every $B\in\Bl$,
\[
\sum_{j,m}|a_{j,m}b_{j,m}|
\ \lsim\ \min\bigl(\|A\|_{\RR},\,\|A\|_{\CC}\bigr)\,\|B\|_{\Bl}.
\]
In particular $\RR+\CC\subset\Bl^K$.
\end{theorem}

\begin{proof}
 Grouping
by rows and applying the Cauchy--Schwarz inequality on each dyadic
block we get that
\begin{align*}
\sum_{j,m}|a_{j,m}b_{j,m}|
&=\sum_{j\ge1}\sum_{n\ge0}\sum_{k\in I_n}|a_k^{\,j}||b_k^{\,j}|\\
&\le\sum_{j\ge1}\sum_{n\ge0}
\Bigl(\sum_{k\in I_n}|a_k^{\,j}|^2\Bigr)^{1/2}
\Bigl(\sum_{k\in I_n}|b_k^{\,j}|^2\Bigr)^{1/2},
\end{align*}
The row half of Lemma~\ref{lem:coeff}(ii) bounds the last factor
by an absolute multiple of $\|B\|_{\Bl}$, uniformly in $j$ and $n$;
this gives the bound with $\|A\|_{\RR}$.

Grouping by columns
($m=l+k$) and using the column half of Lemma~\ref{lem:coeff}(ii) gives
the bound with $\|A\|_{\CC}$. The final assertion follows by
linearity.
\end{proof}

\begin{corollary}\label{cor:l1}
$\Mone\subset\RR\cap\CC\subset\Bl^K$; in particular
$\sum_k\|A_k\|_{S_1}<\infty$ implies $A\in\Bl^K$.
\end{corollary}

\begin{proof}
$\|R_j(A)\|_{\ell(2,1)}\le\|R_j(A)\|_{\ell_1}$ for every $j$, whence
$\|A\|_{\RR}\le\sum_{j,k}|a_k^{\,j}|=\|A\|_{\Mone}$; similarly for
columns.
\end{proof}

The row and column conditions are different, and neither is
necessary.

\begin{proposition}\label{prop:incomparable}
\begin{itemize}
\item[\rm(i)] The matrix $A^{(2)}$ of Theorem~\ref{thm:dichotomy}
satisfies $A^{(2)}\in\RR\setminus\CC$; in particular
$\CC\subsetneq\Bl^K$.
\item[\rm(ii)] There exists $A^{(3)}\in\CC\setminus\RR$; in
particular $\RR\subsetneq\Bl^K$.
\end{itemize}
Consequently neither the row condition nor the column condition is
necessary for membership in $\Bl^K$.
\end{proposition}

\begin{proof}
(i) The only nonzero row of $A^{(2)}$  is 
\[
 R_1(A^{(2)})=(1/k)_{k\ge1},
\]
so 
\[
\|A^{(2)}\|_{\RR}<\infty
\]
by \eqref{eq:harmonic}. Its columns contains only one nonzero entry column $1+k$ contains the single entry $1/k$ so
\[
\|A^{(2)}\|_{\CC}=\sum_{k\ge1}1/k=\infty.
\]
Since $A^{(2)}\in\Bl^K$ by Theorem~\ref{thm:dichotomy}(iii), it follows that $\Bl^K \not\subset \CC$.

(ii) For $m\ge2$ put $\gamma_m=m^{-2}$ and let $A^{(3)}$ having the
entries
\[
(A^{(3)})_k^{\,2^m-k}=\frac{\gamma_m}{k},
\qquad1\le k\le2^{m-1},\quad m\ge2,
\]
all other entries being zero; the $m$-th block of entries occupies the single
column $2^m$ and the rows $2^m-k\in[2^{m-1},2^m-1]$, so distinct
blocks occupy disjoint sets of rows. 

By \eqref{eq:harmonic} we have
\[ 
\|C_{2^m}(A^{(3)})\|_{\ell(2,1)}\lsim\gamma_m,
\]
and hence
\[
\|A^{(3)}\|_{\CC}\lsim\sum_m m^{-2}<\infty.
\]
By
Theorem~\ref{thm:sufficient}, $A^{(3)}\in\Bl^K$. On the other hand
each nonzero row of $A^{(3)}$ contains exactly one entry, so
\[
\|A^{(3)}\|_{\RR}
=\sum_{m\ge2}\gamma_m\sum_{k=1}^{2^{m-1}}\frac1k
\ \ge\ (\ln2)\sum_{m\ge2}\frac{m-1}{m^2}=\infty. \qedhere
\]
\end{proof}

\section{Toeplitz tests and the necessary condition}
\label{sec:necessary}

In this section we prove Theorem~C and deduce the right-hand half of
Theorem~B. The main analytical tool used here is the theorem of de~Leeuw, Kahane and
Katznelson \cite{dLKK}: \emph{there is an absolute constant
$\kappa_0$ such that for every $(x_k)_{k\in\Z}\in\ell_2(\Z)$ there
exists $F\in C(\T)$ with $|\widehat F(k)|\ge|x_k|$ for all $k$ and
$\|F\|_\infty\le\kappa_0\|x\|_{\ell_2}$}; see \cite{Kis} for analytic
versions.

We first record a Bernstein--Schwarz estimate for lacunary sums of
band-limited polynomials. This is a quantitative form of the classical
fact that an analytic function whose smooth dyadic blocks are uniformly
bounded in $H^\infty$ belongs to the Bloch space (see, e.g.,
\cite{Pav}); a matricial analogue for lacunary diagonals is
\cite[Theorem~6]{Po}. Similarly, Lemma~\ref{lem:majorant} below is
essentially known: it follows from \cite{dLKK} by standard
band-limiting, and analytic majorant theorems of this type go back to
\cite{Kis}. We include the short proofs to fix the band geometry and
the constants used in Theorem~\ref{thm:toeplitzdual}.

\begin{lemma}[block polynomials]\label{lem:blocks}
Let $N\ge0$ and let $h_0,\dots,h_N$ be analytic polynomials with
$\max_n\|h_n\|_{H^\infty}\le M$, such that the spectrum of $h_n$ is
contained in $[2^{n-2},2^{n+2})$ for $n\ge2$ and in $[0,4)$ for
$n\in\{0,1\}$. Then
\[
\Bigl\|\sum_{n=0}^Nh_n\Bigr\|_{\mathcal{B}}\ \lsim\ M .
\]
\end{lemma}

\begin{proof}
Write $f=\sum_nh_n$; then $|f(0)|\le2M$. By Bernstein's inequality
\[
\|h_n'\|_{H^\infty}\le2^{n+2}M,
\]
and for $n\ge2$ the polynomial
$h_n'$ vanishes at the origin to order at least $2^{n-2}-1$, so the
Schwarz lemma gives
$\max_{|w|=r}|h_n'(w)|\le2^{n+2}M\,r^{2^{n-2}-1}$ for $0\le r<1$.
Next, for $0<r<1$,
\begin{equation}\label{eq:dyadicsum}
\sum_{j\ge0}2^jr^{2^j}\ \le\ r+2\sum_{i\ge2}r^i\ \le\ \frac{2r}{1-r},
\end{equation}
since for $j\ge1$ the block $(2^{j-1},2^j]$ contains $2^{j-1}$
integers $i$ with $r^i\ge r^{2^j}$, and these blocks are pairwise
disjoint. Substituting $j=n-2$,
\[
\sum_{n\ge2}\max_{|w|=r}|h_n'(w)|
\le16M\,r^{-1}\sum_{j\ge0}2^jr^{2^j}
\le\frac{32M}{1-r},
\]
while $\max_{|w|=r}|h_n'(w)|\le3M$ for $n\in\{0,1\}$. 

Hence
\[
(1-r^2)\max_{|w|=r}|f'(w)|\lsim M
\]
uniformly in $r$, and the claim
follows.
\end{proof}

\begin{lemma}[analytic majorants]\label{lem:majorant}
For every $n\ge0$ and all nonnegative reals $(t_k)_{k\in I_n}$ there
exists an analytic polynomial $g$ such that
\[
|\widehat g(k)|\ge t_k\ \ (k\in I_n),\qquad
\|g\|_{H^\infty}\lsim\Bigl(\sum_{k\in I_n}t_k^2\Bigr)^{1/2},
\]
and the spectrum of $g$ is contained in $[2^{n-2},2^{n+2})$ if
$n\ge2$, and in $I_n\subset[0,4)$ if $n\in\{0,1\}$.
\end{lemma}

\begin{proof}
Normalize $\sum_{k\in I_n}t_k^2=1$. For $n\in\{0,1\}$ take
\[
g(z)=\sum_{k\in I_n}t_kz^k,
\]
so that
\[
\|g\|_{H^\infty}\le\sqrt2.
\]

Let $n\ge2$ and set $N:=2^{n-1}$. Extend $(t_k)$ by zero to $\Z$ and
let $F\in C(\T)$ be given by the de~Leeuw--Kahane--Katznelson theorem,
so that 
\[
\|F\|_\infty\le\kappa_0 \text{ and } |\widehat F(k)|\ge t_k
\]
for
$k\in I_n$. Let $K_M$ be the Fej\'er kernel and let
\[
V:=2K_{2N+1}-K_N
\]
be the de~la~Vall\'ee~Poussin kernel, so that
\[
\widehat V(j)=1 \text{ for } |j|\le N,\quad \widehat V(j)=0
\]
for $|j|\ge2N+2$,
and $\|V\|_{L^1}\le3$.

We define $W(\theta):=\e^{3Ni\theta}V(\theta)$ and
$g:=F*W$. Then $\widehat g=\widehat F\,\widehat W$, so $g$ is a
trigonometric polynomial with spectrum in
\[
[3N-2N-1,\ 3N+2N+1]=[2^{n-1}-1,\ 5\cdot2^{n-1}+1]
\subset[2^{n-2},\ 2^{n+2}),
\]
the inclusion holding for $n\ge2$; in particular all frequencies are
positive and $g$ is an analytic polynomial.

For $k\in I_n$ we have
$|k-3N|\le N$, hence $\widehat W(k)=1$ and
$|\widehat g(k)|=|\widehat F(k)|\ge t_k$. Finally
\[
\|g\|_{H^\infty}\le\|F\|_\infty\|W\|_{L^1}\le3\kappa_0.
\]
\end{proof}

We can now prove Theorem~C. Recall the notation
$S_n(A):=\bigl(\sum_{k\in I_n}\|A_k\|_{S_1}^2\bigr)^{1/2}$, so that
$\|A\|_{\MS}=\sum_nS_n(A)$, and
\[
\sigma(A)=\sup\Bigl\{\sum_{k\ge0}\|A_k\|_{S_1}|\widehat f(k)|:\
\|f\|_{\mathcal{B}}\le1\Bigr\}.
\]

\begin{theorem}[Toeplitz tests characterize $\MS$]
\label{thm:toeplitzdual}
For an upper triangular matrix $A$ the following are equivalent:
\begin{itemize}
\item[\rm(a)] $\sum_{j,m}|a_{j,m}b_{j,m}|<\infty$ for every Toeplitz
matrix $B\in\Bl$;
\item[\rm(b)] $A\in\MS$.
\end{itemize}
Moreover $\sigma(A)\esim\|A\|_{\MS}$.
\end{theorem}

\begin{proof}
If $B=T(f)$ with $f=\sum_kc_kz^k$, then, the diagonals of $B$ being
constant,
\begin{equation}\label{eq:toeplitzpairing}
\sum_{j,m}|a_{j,m}b_{j,m}|
=\sum_{k\ge0}\|A_k\|_{S_1}\,|\widehat f(k)|.
\end{equation}

(b)$\Rightarrow$(a), and $\sigma(A)\lsim\|A\|_{\MS}$. Let
$f\in\mathcal{B}$; by Lemma~\ref{lem:toeplitz}(ii), $T(f)\in\Bl$ with
$\|T(f)\|_{\Bl}=\|f\|_{\mathcal{B}}$, and the first row of $T(f)$ is
$
(\widehat f(k))_{k\ge0}.
$
Lemma~\ref{lem:coeff}(ii) therefore gives
\[
\bigl(\sum_{k\in I_n}|\widehat
f(k)|^2\bigr)^{1/2}\lsim\|f\|_{\mathcal{B}}
\]
uniformly in $n$, and by
blockwise Cauchy--Schwarz in \eqref{eq:toeplitzpairing},
\[
\sum_{k\ge0}\|A_k\|_{S_1}|\widehat f(k)|
\le\sum_{n\ge0}S_n(A)\Bigl(\sum_{k\in I_n}|\widehat f(k)|^2\Bigr)^{1/2}
\lsim\|A\|_{\MS}\,\|f\|_{\mathcal{B}}.
\]

(a)$\Rightarrow$(b), and $\|A\|_{\MS}\lsim\sigma(A)$.

For $k_0\ge0$ we have 
\[
\|z^{k_0}\|_{\mathcal{B}}\le2
\]
by
\eqref{eq:supbound}, and \eqref{eq:toeplitzpairing} applied to
$T(z^{k_0})$ gives 
\[
\|A_{k_0}\|_{S_1}<\infty.
\] 

Consider the linear map
$\Lambda_A:\mathcal{B}\to\ell_1$,
\[
\Lambda_Af:=(\|A_k\|_{S_1}\widehat f(k))_{k\ge0},
\]
well defined by
(a) and \eqref{eq:toeplitzpairing}. The scalar Bloch space is a
Banach space \cite{AS,Pav}, and the coefficient functionals are
continuous on it. By Lemmas~\ref{lem:toeplitz}(ii)
and~\ref{lem:coeff}(i),
\[
|\widehat h(k)|=\|T(h)_k\|_{B(\ell_2)}
\le\e\,\|T(h)\|_{\Bl}=\e\,\|h\|_{\mathcal{B}}
\qquad h\in\mathcal{B},\ k\ge0.
\]
Hence the graph of $\Lambda_A$ is closed, and the closed graph theorem
yields
\begin{equation}\label{eq:CA}
\sum_{k\ge0}\|A_k\|_{S_1}\,|\widehat f(k)|
\le\sigma(A)\,\|f\|_{\mathcal{B}}\qquad f\in\mathcal{B},
\qquad\sigma(A)<\infty .
\end{equation}

 Fix $N\ge0$. For each $0\le n\le N$
with $S_n(A)>0$ let us take
\[
t_k:=\|A_k\|_{S_1}/S_n(A)
\]
for $k\in I_n$. 

It follows 
that $\sum_{k\in I_n}t_k^2=1$, and let $g_n$ be the polynomial of
Lemma~\ref{lem:majorant}. If $S_n(A)=0$ we take $g_n:=0$. 

For $i\in\{0,1,2,3\}$ define the index sets
\[
E_i:=\{0\le n\le N:\ n\equiv i\ (\mathrm{mod}\ 4)\}.
\]
By construction, the polynomials $\{g_n\}_{n \in E_i}$ have pairwise disjoint spectra. Specifically, $\operatorname{supp}(\widehat{g_n}) \subset [2^{n-2}, 2^{n+2})$ for $n \ge 2$, yielding mutually disjoint supports whenever $|n - m| \ge 4$.

Additionally, $\operatorname{supp}(\widehat{g_0}) \subset \{0,1\}$ and $\operatorname{supp}(\widehat{g_1}) \subset \{2,3\}$, ensuring disjointness from the supports of all $g_n$ with $n \ge 4$.

Fix $i$ and set $f_i:=\sum_{n\in E_i}g_n$; by
Lemmas~\ref{lem:blocks} and~\ref{lem:majorant},
$\|f_i\|_{\mathcal{B}}\lsim1$. By the spectral disjointness within the
class, $\widehat{f_i}(k)=\widehat{g_n}(k)$ for $k\in I_n$, $n\in E_i$,
so \eqref{eq:CA} yields
\[
\sum_{n\in E_i}S_n(A)
=\sum_{n\in E_i}\sum_{k\in I_n}t_k\,\|A_k\|_{S_1}
\le\sum_{k\ge0}|\widehat{f_i}(k)|\,\|A_k\|_{S_1}
\le\sigma(A)\,\|f_i\|_{\mathcal{B}}\lsim\sigma(A).
\]
Summing over the four classes and letting $N\to\infty$ gives
$\|A\|_{\MS}\lsim\sigma(A)$.
\end{proof}

\begin{theorem}\label{thm:necessary}
$\Bl^K\subset\MS$, and the inclusion is strict.
\end{theorem}

\begin{proof}
Let $A\in\Bl^K$. Then 
\[
\sum_{j,m}|a_{j,m}b_{j,m}|<\infty
\] for every
$B\in\Bl$, and in particular for every Toeplitz matrix $B\in\Bl$. This
is condition (a) of Theorem~\ref{thm:toeplitzdual}, thus $A\in\MS$ and
\[
\|A\|_{\MS}\lsim\sigma(A).
\]
The inclusion is strict: by
Theorem~\ref{thm:dichotomy}, the matrix $A^{(1)}$ lies in $\MS$ but
not in $\Bl^K$.
\end{proof}

\begin{corollary}[two-sided estimate]\label{cor:sandwich}
\[
\RR+\CC\ \subset\ \Bl^K\ \subset\ \MS ,
\]
where the second inclusion is strict, and each of $\RR$, $\CC$ alone
is a proper subset of $\Bl^K$.
\end{corollary}

\begin{proof}
The proof follows direct from  Theorems~\ref{thm:sufficient} and~\ref{thm:necessary} with
Proposition~\ref{prop:incomparable}.
\end{proof}

\begin{remark}\label{rem:notoeplitz}
No nonzero Toeplitz matrix can belong to $\Bl^K$. A nonzero Toeplitz matrix necessarily contains a diagonal that is a nonzero multiple of the identity, giving it an infinite trace norm and thereby failing the necessary condition of Theorem~\ref{thm:necessary}. Thus, the K\"othe dual of $\Bl$ contains exclusively matrices with strongly decaying diagonals, standing in contrast to $\Bl$ itself, which contains $\{T(f):f\in\mathcal{B}\}$ as an isometric Toeplitz copy of the scalar Bloch space.
\end{remark}

\begin{remark}\label{rem:scalarcomparison}
Theorem~\ref{thm:toeplitzdual} is the exact matricial analogue of the
Anderson--Shields identity \eqref{eq:AS}: the map
$A\mapsto(\|A_k\|_{S_1})_{k\ge0}$ identifies the Toeplitz-tested
K\"othe dual of $\Bl$ with the scalar multiplier space
$(\mathcal{B},\ell_1)=\ell(2,1)$. Theorem~\ref{thm:dichotomy} and
Corollary~\ref{cor:impossible} show that this description does not
survive the passage from Toeplitz tests to all of $\Bl$, and that no
repair within the class of diagonal descriptions is possible.
\end{remark}

\section{Open problems}\label{sec:problems}

\begin{problem}\label{prob:char}
Characterize $\Bl^K$ intrinsically. In particular, does
$\Bl^K=\RR+\CC$ hold, with equivalence of the natural norms? By
Corollary~\ref{cor:impossible}, any characterization must take into
account the positions of the entries within the diagonals.
\end{problem}

\begin{problem}\label{prob:other}
In \cite[Theorem~4.4]{MMPP} the K\"othe duals of $\mathcal{I}(\ell_2)$
and of $\Blz$ are identified with the largest solid subspaces
$s(\Bl)$ and $s(\mathcal{I}(\ell_2))$, respectively. Give explicit,
entrywise descriptions of these solid parts, and determine the K\"othe
dual of the Bergman--Schatten space $L_a^1(D,\ell_2)$ of
\cite{MPPP,PP}. Remark~\ref{rem:littlebloch} yields a negative result: $\MS\not\subset s(\mathcal{I}(\ell_2))$, so a trace-norm dyadic
description is ruled out for $s(\mathcal{I}(\ell_2))$ as well. We note
that K\"othe duality, being an order-theoretic notion, interacts with
the \emph{solid core} of a space rather than with its solid hull, a
distinction that must be respected when transferring scalar arguments.
\end{problem}

\begin{problem}\label{prob:quant}
Endowed with the norm
$\|A\|_{K}:=\sup\{\sum_{j,m}|a_{j,m}b_{j,m}|:\|B\|_{\Bl}\le1\}$, which
is finite on $\Bl^K$ by the closed graph theorem, is $\Bl^K$
isomorphic, as a Banach space, to a classical space of matrices or
sequences? What are the optimal constants in
Theorem~\ref{thm:sufficient} and Theorem~\ref{thm:toeplitzdual}?
\end{problem}

\end{document}